\documentclass[11pt,twoside]{amsart}
\usepackage{amsmath, amsthm, amscd, amsfonts, amssymb, float, graphicx, color, xcolor, soul }
\usepackage[bookmarksnumbered, colorlinks, plainpages]{hyperref}
\usepackage{graphicx}
\usepackage{epstopdf}
\usepackage{rotating}
\input xy

\newtheorem{theorem}{Theorem}
\newtheorem{remark}{Remark}
\newtheorem{corollary}{Corollary}
\newtheorem{lemma}{Lemma}
\newtheorem{question}{Question}
\newtheorem{proposition}{Proposition}
\newtheorem{definition}[equation]{Definition}
\newtheorem{conjecture}[equation]{Conjecture}
\usepackage[bottom]{footmisc}

\newcommand{\DCI}{{\rm DCI}}
\newcommand{\Cay}{{\rm Cay}}
\newcommand{\BCI}{{\rm BCI}}

\newcommand{\CI}{{\rm CI}}
\newcommand{\Sym}{{\rm Sym}}
\newcommand{\BCay}{{\rm BCay}}
\newcommand{\Spec}{{\rm Spec}}
\newcommand{\Aut}{{\rm Aut}}

\date{}
\def\pn{\par\noindent}
\begin{document}
\title{On $\BCI$-groups and $\CI$-groups‌}
\author{Asieh Sattari, Majid Arezoomand and Mohammad A. Iranmanesh$^*$}
\thanks{{\scriptsize
\hskip -0.4 true cm MSC(2010): Primary 05C25; Secondary: 05C60, 05E18
$*$Corresponding author}}
\maketitle
\begin{abstract}
Let $G$ be a finite group and
$S$ be a subset of $G.$ A bi-Cayley graph $\BCay(G,S)$ is a simple and an undirected
graph with vertex-set $G\times\{1,2\}$ and edge-set
$\{\{(g,1),(sg,2)\}\mid g\in G, s\in S\}$. A bi-Cayley graph $\BCay(G,S)$ is called a
$\BCI$-graph if for any bi-Cayley graph $\BCay(G,T)$, whenever
$\BCay(G,S)\cong\BCay(G,T)$ we have $T=gS^\sigma$ for some $g\in G$
and $\sigma\in\Aut(G).$ A group $G$ is called a $\BCI$-group if
every bi-Cayley graph of $G$ is a $\BCI$-graph. In this paper, we
showed that every $\BCI$-group is a $\CI$-group, which gives a positive answer to a conjecture
proposed by Arezoomand and Taeri in \cite{arezoomand1}. Also we
proved that there is no any non-Abelian $4$-$\BCI$-simple group. In
addition all $\BCI$-groups of order $2p$, $p$ a prime, are characterized.
\end{abstract}

\section{Introduction}
Throughout this paper all graphs and groups are finite. Graphs are simple and undirected, where by a simple
graph we mean a graph with no multiple edges or loops. Our notation are standard and can be found in \cite{Mil}

Let $G$ be a group and $S$ be a
subset of $G$ such that $1\notin S$ and $S=S^{-1}$. Then $\Cay(G,S)$
is a simple and undirected graph with vertex set $G$ and edge set $ E=\{\{g,sg\}|s\in S, g\in G\}$.
A fundamental problem that about 50 years ago arose, is
\emph{Isomorphism Problem} for two Cayley graphs. That is, when two
Cayley graphs $\Cay(G,S)$ and $\Cay(H,T)$ are isomorphic? It follows
quickly from the definition that for any automorphism
$\alpha\in\Aut(G)$, the graphs $\Cay(G,S)$ and $\Cay(G,S^\alpha)$
are isomorphic, namely, $\alpha$ induces an isomorphism between
these graphs. Such an isomorphism is called a {\it Cayley
isomorphism}. In 1967, Ad\'{a}m \cite{Adam} conjectured that two Cayley
graphs over the cyclic group $\mathbb{Z}_n$ are isomorphic if and only
if there is a Cayley isomorphism which maps one to the other. Soon
afterwards, Elspas and Turner \cite{Elspas} found the counterexample
for $n=8$. This also motivated the following definition. A Cayley
graph $\Cay(G,S)$ is a $\CI$ graph if whenever $\Cay(G,S)\cong
\Cay(G,T)$ for some subset $T$ of $G,$ then $T = S^\alpha$ for some
$\alpha\in\Aut(G)$. The group $G$ is an $m$-$\CI$-group if every
Cayley graph over $G$ of valency at most $m$ is a $\CI$-graph, and
$G$ is a $\CI$-group if every Cayley graph over $G$ is a
$\CI$-graph. The problem of classifying finite $\CI$-groups is
still open \cite{Li,MS,Spiga}. Let $G$ be a finite group and
$S$ be a subset of $G.$ A bi-Cayley graph $\BCay(G,S)$ is an undirected graph
with vertex-set $G\times\{1,2\}$ and edge-set
$\{\{(g,1),(sg,2)\}\mid g\in G, s\in S\}.$

In 2008, motivated by the
concepts of $\CI$ graph, $m$-$\BCI$-group and $\CI$-group, Xu et al.
\cite{j.xu} introduced the concepts $\BCI$-graph, $m$-$\BCI$-group
and $\BCI$-group, respectively. We say that a bi-Cayley graph
$\BCay(G,S)$ is a $\BCI$-graph if whenever $\BCay(G,S)\cong
\BCay(G,T)$ for some subset $T$ of $G$, then $T = gS^\alpha$ for
some $g\in G$ and $\alpha\in\Aut(G)$. The group $G$ is
an $m$-$\BCI$-group if every bi-Cayley graph over $G$ of valency at
most $m$ is a $\BCI$-graph, and $G$ is a $\BCI$-group if every
bi-Cayley graph over $G$ is a $\BCI$-graph. The theory of
$\BCI$-graphs and $\BCI$-groups is less developed as in the case
of $\CI$-graphs and $\CI$-groups.
Jin and Liu in a series of papers \cite{jin1,jin2,jin4}
obtained several basic properties about $\BCI$-graphs and
$\BCI$-groups. $\BCI$-graphs and
$\BCI$-groups are studied by Koike et.al. in \cite{koike1,koike2,koike3,koike4}
and by the second author in \cite{arezoomand1, arezoomand2}.

Our primary motivation by studying $\BCI$-graphs and $\BCI$-groups is that these objects
can bring new insight into the old problem of characterizing
$\CI$-groups. In \cite{arezoomand1} it is conjectured that every
$\BCI$-group is a $\CI$-group.

This paper is organized as follows. In Section 2, we proved that every $\BCI$-group is
a $\CI$-group. In Section 3, we classify cyclic $\BCI$-$p$-groups and we will show that there
is no any non-Abelian $4$-$\BCI$ simple group. In Section 4, $\BCI$-groups of order $2p$
are considered and we prove that $\mathbb{Z}_{2p}$ is a $\BCI$-group. Indeed we show that
$\mathbb{Z}_{2p}$ is a $\BCI$-group and $\BCI$-groups of order $2p$ where $p$ is
a prime are characterized.

\section{The relation between $\BCI$-groups and $\CI$-groups }
%\textcolor[rgb]{0.00,0.00,0.98}{Dear Prof. Iranmanesh, some proofs ended by square and some not!}
In this section we prove that set of
finite $\BCI$-groups is a subset of the set of finite
$\CI$-groups. It causes to shift many properties from
$\CI$-groups to $\BCI$-groups. We mention some of them here.

\begin{theorem}\label{t1}Every finite $\BCI$-group is a $\CI$-group.
\end{theorem}
\begin{proof} Let $G$ be a $\BCI$-group and $S,T \subseteq G$ such that
$S=S^{-1}$ and $T = T^{-1}$ and $1\notin S\cap T.$ Suppose that
$\Cay(G,S)\cong \Cay(G,T).$ We prove that $S=T^{\alpha}$ for
$\alpha \in \Aut(G).$ By \cite[Lemma 4.7]{arezoomand1}, we have
$\BCay(G,S \cup \{1\})\cong \BCay(G,T \cup \{1\}).$ As $G$ is a
$\BCI$-group, there exist $ g\in G$ and $\alpha\in\Aut(G)$ such
that $S\cup \{1\}=g(T\cup\{1\})^{\alpha}.$ If $g=1$ then
$S=T^{\alpha}.$ So we may assume that $g\neq 1.$ We will prove the
theorem by induction on $|S|.$

If $|S|=1$ then $S=\{s_0\}$ and $T=\{t_0\}$ for some $s_0, t_0\in G$. Then
$\{s_0, 1\}= g\{t_0, 1\}^{\alpha}= \{gt_0^{\alpha}, g\}$.
Since $g\neq 1$, we have $g=s_0$ and $gt_0^{\alpha}=1$. Therefore,
$s_0t_0^{\alpha}=1.$ This implies that $t_0^{\alpha}=s_0^{-1}$.
Since $S=S^{-1}$ we conclude that $s_0^{-1}=s_0=t_0^{\alpha}.$ So $S=T^{\alpha}$ as desired.

Assume that the statement is true for $|S|<n $. Let $|S|=n$, $S=\{s_1,s_2,...,s_n\}$ and $T=\{t_1,t_2,...,t_n\}$.
Since $S\cup \{1\}= g_0(T\cup\{1\})^{\alpha}$ for some $g_0\in G$, we conclude that $\{s_1,s_2,...,s_n,1\}= \{g_0t_1^{\alpha},
..., g_0t_n^{\alpha}, g_0\}$. As $g_0\neq 1$ and $g_0\in S$ we can assume that $g_0=s_n$ and $g_0{t_n}^{\alpha}=1.$ Hence
$t_n^{\alpha}= g_0^{-1}$ which implies that $t_n^{\alpha}=s_n^{-1}.$ Set $S_0=S\setminus\{s_n^{-1}\}$
and $T_0=T\setminus \{t_n\}$. Then we have $S_0\cup\{1\}= g_0(T_0\cup\{1\})^{\alpha} $ and $|S_0|<n$. By the
induction hypothesis, $S_0= {T_0}^{\alpha}$ and therefore, $S=S_0\cup \{s_n^{-1}\} = T_0^{\alpha}\cup\{t_n^{\alpha}\}=T^{\alpha}$.
\end{proof}

\begin{corollary}\label{p1} Let $G$ be a finite group, $1\leq m <\mid G\mid$
and $G$ has the $(m+1)$-$\BCI$ property. Then $G$ has the $m$-$\CI$ property.
\end{corollary}

By Theorem \ref{t1} and Corollary \ref{p1} we can find many results which are obtained in $\CI$-groups.
The best list of $\CI$-groups is due to Li in \cite{Li}. It should be mentioned that their proof was incomplete,
but this was corrected by Dobson in \cite{DMS}. Here for Sylow subgroups we mention two remarkable results below.

\begin{proposition}
Let $G$ be a $\BCI$-group of odd order. Then a Sylow 3-subgroup is $\mathbb{Z}_{3^k}$, $k=1,2,3$ and
if $p\neq 3$, then Sylow $p$-subgroups are elementary Abelian. Furthermore, $G$ is an Abelian group, or $G$
has an Abelian normal subgroup of index 3.
\end{proposition}
\begin{proof} This is an immediate consequence of Theorem \ref{t1} and \cite[Theorem 8.1]{Li}.
\end{proof}

\begin{proposition}
Suppose that $G$ is a finite group with the 5 or 6-$\BCI$ property. Then a Sylow 2-subgroup of $G$ is elementary
Abelian, cyclic, or generalized quaternion.
\end{proposition}
\begin{proof} By Corollary \ref{p1}, $G$ is a group with 4 or 5-$\CI$ property. So by \cite[Lemma 3.1]{CLi}, the statement is true.
\end{proof}

\begin{proposition}
Suppose that $G=\mathbb{Z}^{n}_p$, $p>2$ a
prime number, with $n\geq 2p + 3$. Then $G$ is not a
$\BCI$-group.
\end{proposition}
\begin{proof} It follows from Theorem \ref{t1} and \cite[Theorem 1]{So}.
\end{proof}
\begin{definition}\cite[Definition 5.]{D}\label{def:1} Let $M$ be an Abelian group such that every
Sylow $p$-subgroup of $M$ is elementary abelian. Denote the largest order of any element
of $M$ by $\exp(M).$ Let $n\in \{2, 3, 4, 8\}$ be relatively prime to $|M|.$ Set
$E(n, M)=\mathbb{Z}_n \ltimes_{\phi} M,$ where if $n$ is even
then $\phi(g) = g^{-1}$, while if $n = 3$ then $\phi(g) = g^{\ell}$, where $\ell$ is an integer satisfying
$\ell^3\equiv 1 (\mod \exp(M))$ and $(\ell(\ell-1), \exp(M))=1$.
If $M = \mathbb{Z}_p$, and 3 divides $(p-1)$ then $E(3, \mathbb{Z}_p)$ is the nonabelian group of order $3p$, which we
denote by $F_{3p}$ (as this group is a Frobenious group). Similarly, $E(2, \mathbb{Z}_n)$ is the dihedral
group of order $2n$.
%The next result is a combination of results of Li, Lu, and Palfy [9],
%and Somlai [11], and lists all possible CI-groups with respect to graphs. Not every group
%in this result is known to be a CI-group with respect to graphs - see [6] for a recent list
%of the known CI-groups with respect to graphs.
\end{definition}
The following corollary is an imidiate consequence of Definition \ref{def:1}, \cite[Theorem 6.]{D} and Theorem \ref{t1}.
%\begin{theorem}\cite[Theorem 6.]{D} Let $G$ be a $\CI$-group with respect to graphs.
%\begin{itemize}
%\item[(a)] If there is not any elements of order 8 or 9 in $G$, then $G =
%H_1\times H_2\times H_3$, where the orders of $H_1$, $H_2$, and $H_3$ are pairwise coprime, and

%\begin{itemize}
%\item[$(i)$] $H_1$ is an Abelian group, and each Sylow subgroup of $H_1$
%is isomorphic to $\mathbb{Z}_p^k$ for $k<2p+3$ or $\mathbb{Z}_4$;
%\item[$(ii)$] $H_2$ is one of the groups $E(M, 2)$, $E(M, 3)$, $E(M, 4)$, $A_4$, $Q_8$, or 1;
%\item[$(iii)$] $H_3$ is one of the groups $D_{10}$, or 1.
%\end{itemize}

%\item[(b)] If $G$ has elements of order 8, then $G \cong E(M, 8)$.
%\item[(c)] If $G$ contains elements of order 9, then $G$ is one of the
%groups $\mathbb{Z}_9 \ltimes \mathbb{Z}_2$, $\mathbb{Z}_9 \ltimes \mathbb{Z}_4$, $\mathbb{Z}_2^2\ltimes
%\mathbb{Z}_9$, or $\mathbb{Z}_9\times\mathbb{Z}_2^{n}$ with $n\leq 5$.
%\end{itemize}
%\end{theorem}

%\textcolor[rgb]{0.00,0.00,0.98}{Dear Prof. Iranmanesh, (1) We must change this result to a Corollary, (2) I suggest to move it after Theorem \ref{t1}, (3) this a very very careless %copy from the reference.}
%\begin{proof} By Theorem \ref{cyclicp}, $\mathbb{Z}_8$ is not a $\BCI$-group. By Theorem \ref{t1} and
%\cite[Corollary 13]{D}, the statement is true.
%\end{proof}

\begin{corollary} Let $G$ be a finite $\BCI$-group.
\begin{itemize}
\item[(a)] If there is not any elements of order 8 or 9 in $G$, then $G=H_1\times H_2\times H_3$,
where the orders of $H_1$, $H_2$, and $H_3$ are pairwise coprime, and
\begin{itemize}
\item[$(i)$] $H_1$ is an Abelian group, and each Sylow $p$-subgroup of $H_1$
is isomorphic to $\mathbb{Z}_p^k$ for $k<2p+3$ or $\mathbb{Z}_4$;
\item[$(ii)$] $H_2$ is isomorphic to one of the groups $E(2, M)$, $E(M, 4)$, $Q_8$, or 1;
\item[$(iii)$] $H_3$ is isomorphic to one of the groups $E(3, M)$, $A_4$, or 1.
\end{itemize}

\item[(b)] If $G$ has elements of order 8, then $G \cong E(8, M)$ or $\mathbb{Z}_8$.
\item[(c)] If $G$ contains elements of order 9, then $G$ is one of the
groups $\mathbb{Z}_2 \ltimes \mathbb{Z}_9$, $\mathbb{Z}_4 \ltimes \mathbb{Z}_9$, $\mathbb{Z}_9\ltimes \mathbb{Z}_2^2$, or $\mathbb{Z}_2^n\times\mathbb{Z}_9$ with $n\leq 5$.
\end{itemize}
\end{corollary}

By considering Theorem \ref{t1}, one can check that which properties of $\CI$-groups arise in $\BCI$-groups.
For example, by \cite[Lemma 8.2]{Li} we know that if $G$ is a $\CI$-group,
then every subgroup of $G$ is a $\CI$-group. The following question therefore arises.
%\textcolor[rgb]{0.00,0.00,0.98}{Dear Prof. Iranmanesh, I suggest
%we change the question to "Which properties of $\CI$-groups arise in $\BCI$-groups?", and explain that the following is a very partial answer to the question.}

\begin{question}\label{Q1}
Which properties of $\CI$-groups arise in $\BCI$-groups?
\end{question}
It is proved that the subgroup of a $\CI$-group is also a $\CI$-group. Here we pose the following conjecture.
\begin{conjecture}\label{conj:1}
Let $G$ be a $\BCI$-group and $H$ be a subgroup of $G$. Then $H$ is a $\BCI$-group.
\end{conjecture}
In the following lemma we give a partial answer to the Conjecture \ref{conj:1}.
%\textcolor[rgb]{0.00,0.00,0.98}{Dear Prof. Iranmanesh, since the CI-version of the result is that "any subgroup of a CI-group is a CI-group", we have to give a counterexample that we can not replace "characteristic subgroup" with "subgroup", or prove it, or explain that we have not any counterexample in our hand.}
\begin{lemma}\label{char-sub}
Let $G$ be a finite $BCI$-group and $H$ be a characteristic subgroup
of $G$. Then $H$ is also a $BCI$-group.
\end{lemma}
\begin{proof} Let $S,T\subseteq H$ and $\BCay(H,S)\cong\BCay(H,T)$. Then \cite[Lemma 3.5]{arezoomand3} implies that
$\BCay(G,S)\cong\BCay(G,T)$. Since $G$ is a $\BCI$-group, there exist
$\alpha\in\Aut(G)$ and $g\in G$ such that $T=gS^\alpha$. Since $H$
is a characteristic subgroup of $G$, we find that $\alpha|_H\in\Aut(H)$. Since
$S$ and $T$ are subsets of $H$, we conclude that $g\in H$ which
means that $H$ is a $\BCI$-group.
\end{proof}

The following corollary for the direct product of two $\BCI$-groups can obtain from
Lemma \ref{char-sub}.

\begin{corollary}\label{coprime}
Let $G$ and $H$ be two finite groups and $(|G|,|H|)=1$. If $G\times
H$ is a $\BCI$-group then both $G$ and $H$ are $\BCI$-groups. In
particular, every Sylow subgroup of a finite nilpotent $\BCI$-group
is a $\BCI$-group.
\end{corollary}

\begin{remark}
The converse of Corollary \ref{coprime} is not true, because by
\cite[p. 28]{LPX}, $\mathbb{Z}_{27}$ is not a $\CI$-group and by
\cite[Lemma 3.2]{BF}, every subgroup of a $\CI$-group is a
$\CI$-group. Therefore $\mathbb{Z}_{54}$ is not a $\CI$-group and by
Theorem \ref{t1}, it is not a $\BCI$-group. While $\mathbb{Z}_9$ and
$\mathbb{Z}_6$ are both $\BCI$-groups by Theorem \ref{cyclicp} and
Theorem \ref{z2p}.
\end{remark}

%\textcolor[rgb]{0.00,0.00,0.98}{I did not find any referring to this lemma in the paper. So I suggest to remove it, because really it is not a very strong result.}
%\begin{lemma}\label{cr}
%Let $G$ be a $\BCI$-group and $H,K\leq G$ be of the same order. Then there exists $\alpha\in\Aut(G)$ such that $K=H^\alpha$.
%\end{lemma}
%\begin{proof} Let $|H|=|K|=n$. Then we have $|G:H|=|G:K|$ and $\BCay(H,H)\cong\BCay(K,K)\cong K_{n,n}$. Thus $\BCay(G,H)\cong\BCay(G,K)$.
%Since $G$ is a $\BCI$-group, there exist $g\in G$ and $\alpha\in\Aut(G)$ such that $K=gH^\alpha$. Since $1\in H$, we have $g\in K$, which
%implies that $K=H^\alpha$.
%\end{proof}

It is proved in \cite{jin2} that the only finite simple non-abelian 3-$\BCI$-group is $A_5$. In the following theorem,
we prove that there is no any simple non-abelian 4-$\BCI$ group.
\begin{theorem}
There is no any non-Abelian 4-$\BCI$ simple group.
\end{theorem}
\begin{proof} Let $G$ be a finite non-Abelian simple group. It is proved in \cite{jin2}
that $G$ is a 3-$\BCI$-group if and only if
$G\cong A_5$. Since any 4-$\BCI$-group is a 3-$\BCI$ group, it is enough to prove that $A_5$
is not a 4-$\BCI$ group.

Let $G\cong A_5$, and $a=(1~2~3)$ and $b=(1~2~3~4~5).$
Assume that $\BCay(G,S)\cong \BCay(G,S^{-1})$ where $S=\{1,a,b,ab\}.$
We will show that $S^{-1}\neq gS^{\alpha}$ for $g\in G$ and $\alpha \in \Aut(G)$. As
$1\in S$, we conclude that $g\in S^{-1}$. So we have the following cases:

Case I. $g=1.$ Then $S^{-1}=S^{\alpha}$. Therefore, we find that $\{a^{-1},b^{-1},(ab)^{-1}\}=\{a^{\alpha},b^{\alpha},(ab)^{^{\alpha}}\}$. So we have
$a^{\alpha}=a^{-1}$. If $b^{\alpha}=b^{-1}$ and $(ab)^{\alpha}=(ab)^{-1}$, then $b^{-1}a^{-1}=(ab)^{-1}=(ab)^{\alpha}=a^{\alpha}b^{\alpha}=a^{-1}b^{-1}$
a contradiction. Let $b^{\alpha}=(ab)^{-1}$ and $(ab)^{\alpha}=b^{-1}.$ Then $b^{-1}=(ab)^{\alpha}=a^{\alpha}b^{\alpha}=a^{-1}b^{-1}a^{-1},$
which is a contradiction.

Case II. $g=a^{-1}$, then $aS^{-1}=S^{\alpha}$. In this case $\{a,1,ab^{-1},a(ab)^{-1}\}=\{1,a^{\alpha},b^{\alpha},(ab)^{^{\alpha}}\}$ and therefore
$a^{\alpha}=a$. If $b^{\alpha}=ab^{-1}$ and $(ab)^{\alpha}=a(ab)^{-1}$, then
$ab^{-1}a^{-1}=(ab)^{\alpha}=a^{\alpha}b^{\alpha}=a^{2}b^{-1}.$ So $ab^{-1}=b^{-1}a^{-1}$ a contradiction. Let $b=a(ab)^{-1}$
and $(ab)^{\alpha}=ab^{-1}$. Hence $ab^{-1}=(ab)^{\alpha}=a^{\alpha}b^{\alpha}=aab^{-1}a^{-1}$ which implies that
$b^{-1}=ab^{-1}a^{-1}$ a contradiction.

Case III. $g=b^{-1}$, then we have $bS^{-1}=S^{\alpha}$. Hence
$\{b,ba^{-1},1,a^{-1}\}=\{1,a^{\alpha},b^{\alpha},(ab)^{^{\alpha}}\}$. So
$a^{\alpha}=a^{-1}$. Now if $b^{\alpha}=b$ and $(ab)^{\alpha}=ba^{-1},$ then we have
$ba^{-1}=a^{\alpha}b^{\alpha}=a^{-1}b$. It is a contradiction. If
$b^{\alpha}=ba^{-1}$ and $(ab)^{\alpha}=b$ we have
$b=a^{\alpha}b^{\alpha}=a^{-1}ba^{-1}$ which is another contradiction.

Case IV. $g=(ab)^{-1}$, then we have $abS^{-1}=S^{\alpha}.$ In this case we find that $\{ab,aba^{-1},a,1 \}=\{1,a^{\alpha},b^{\alpha},(ab)^{\alpha}\}$.
Hence $a^{\alpha}=a$. Assume that $b^{\alpha}=ab$ and $(ab)^{\alpha}=aba^{-1}$. On the other hand $a^{\alpha}b^{\alpha}=a^{2}b.$
Hence $ab=ba^{-1}$ a contradiction. In case $b^{\alpha}=aba^{-1}$
and $(ab)^{\alpha}=ab$ we have
$a^{\alpha}b^{\alpha}=a^{2}ba^{-1}$ which implies $aba^{-1}=b$ a contradiction.

Hence $A_5$ is not a 4-$\BCI$-group and the proof is complete.
\end{proof}

%*****************************************************************************************************************************************

\section{$\BCI$-groups of order $p^k$ and $2p$}
It is well-know that $\mathbb{Z}_p$, $p$ a prime, is a $\CI$-group. Also it is proved in
\cite[Corollary 4.9]{arezoomand1} that it is a $\BCI$-group. This motivates to study finite groups which are both $\BCI$ and $\CI$-group.
Let $\mathcal{BC}$ denotes the class of finite groups $G$ which are both $\BCI$ and $\CI$-groups.
Answering to this question that which groups are in $\mathcal{BC}$?
For a prime number $p$, and a positive integer $k$, we will classify finite cyclic
$\BCI$-group of order $p^k$ and $\BCI$-group of order $2p$.

\begin{theorem}\label{cyclic ci}({\cite{M1,M2}Muzychuk}) A cyclic group of order $n$ is a $\CI$-group if
and only if either $n\in \{8,9,18\}$ or $n=k, 2k$ or $4k$ where $k$
is odd square-free.
\end{theorem}
In the following theorem, we classify finite cyclic
$\BCI$-$p$-groups:

\begin{theorem}\label{cyclicp}
A finite cyclic $p$-group $G$ is $\BCI$-group if and only if $G$
isomorphic to one of the groups $\mathbb{Z}_2,\mathbb{Z}_4,\mathbb{Z}_3,\mathbb{Z}_9,\mathbb{Z}_p,$ where $p\geq 5$.
\end{theorem}
\begin{proof} Let $G=\langle a\rangle\cong\mathbb{Z}_{p^k}$ for some prime $p$
and positive integer $k$ be a $\BCI$-group. By Theorem \ref{t1},
$G$ is a $\CI$-group. Hence by Theorem \ref{cyclic ci}, $G$ is isomorphic to one
of the groups $\mathbb{Z}_2,\mathbb{Z}_4,\mathbb{Z}_3,\mathbb{Z}_9,\mathbb{Z}_p$, where $p\geq 5$.

To complete the proof it is enough to show that the groups $\mathbb{Z}_2,\mathbb{Z}_4,\mathbb{Z}_3,\mathbb{Z}_9,\mathbb{Z}_p,$ where $p\geq 5$
are $\BCI$-groups. By \cite[Corollary 4.9]{arezoomand1} the groups $\mathbb{Z}_9$ and $\mathbb{Z}_p$, $p$ a
prime, are $\BCI$-group. Let $\Gamma=\BCay(\mathbb{Z}_4, S)$ for some
subset $S$ of $\mathbb{Z}_4$. If $|S|\leq 3,$ then by \cite{jin1} or \cite[Theorem 1.1]{koike2} $\Gamma$ is a
$\BCI$-graph. Hence we may assume that $S=\mathbb{Z}_4$. In this case, obviously $\Gamma$ is a $\BCI$-graph.
This completes the proof.

The following corollary which is an immediate consequence of Lemma \ref{char-sub} gives
us some restriction on finite cyclic $\BCI$-groups.
\end{proof}

\begin{corollary} Let $G\cong\mathbb{Z}_n$ be a $\BCI$-group. Then
$n=2^{i}3^{j}p^{\alpha_1}_1\cdots p^{\alpha_k}_k$, $0\leq
i\leq2$, $0\leq j\leq 2$, $0\leq \alpha_t\leq 1$, for $t=1,\ldots,k$.
\end{corollary}

As a consequence of Theorem \ref{cyclicp}, we determine dihedral $\CI$-groups of order $2p^k$,
where $p$ is a prime and $k\geq 1$ is an integer.

\begin{corollary} Let $p\geq 3$ be a prime and $k\geq 1$ be an integer. Then $D_{2p^k}$ is a $\CI$-group if and only if
$p\geq 5$ and $k=1$ or $(p,k)\in\{(2,1),(3,1),(3,2)\}$.
\end{corollary}
\begin{proof} Let $D_{2p^k}$ be a $\CI$-group. Then Theorem \ref{cyclicp}
and \cite[Corollary 4.9]{arezoomand2} imply that $p\geq 5$ and $k=1$
or $(p,k)\in\{(2,1),(2,2),(3,1),(3,2)\}$. On the other hand, it is
well-known that $D_8$ is not a $\CI$-group. This proves one direction.

By \cite{Babai} $D_{2p}$, $p\geq 3$ a prime is a $\CI$-group. Also by
\cite{Godsil} $D_4\cong\mathbb{Z}_2\times\mathbb{Z}_2$ is a $\CI$-group.
Furthermore, $D_{18}$ is a $\CI$-group by \cite{DMS}. This completes the proof.
\end{proof}

Before turning to prove that the group $\mathbb{Z}_{2p}$, $p$ an odd prime, is a
$\BCI$-group, we need to prove some lemmas.
%\end{proof}

\begin{lemma}\label{comp}
The bi-Cayley graph, $\BCay(G,S)$, is a $\BCI$-graph if and only if $\BCay(G, G\setminus S)$
is a $\BCI$-graph.
\end{lemma}
\begin{proof} Since $G\setminus(G\setminus S)=S$, it is enough to prove the
direction $``\Rightarrow"$. To this end, suppose that
$\Gamma=\BCay(G,S)$ is a $\BCI$-graph and $\Sigma=\BCay(G,G\setminus
S)$. Let $\varphi\in\Sym(V(\Sigma))$ where
$\{G\times\{1\},G\times\{2\}\}^\varphi=\{G\times\{1\},G\times\{2\}\}$
and $\varphi^{-1}R_G\varphi\leq\Aut(\Sigma)$. By \cite[Theorem
C]{arezoomand1}, it is enough to prove that $R_G$ and
$\varphi^{-1}R_G\varphi$ are conjugate in $\Aut(\Sigma)$ and
$(G\setminus S)^{-1}=g(G\setminus S)^\alpha$ for some $g\in G$ and
$\alpha\in\Aut(G)$.

First, we claim that $\varphi^{-1}R_G\varphi\leq\Aut(\Gamma)$.
Let $\rho_g$ be an arbitrary element of $R_G$ and $x,y\in G$. Since
$\{G\times\{1\},G\times\{2\}\}$ is $\varphi$-invariant,
$(x,1)^{\varphi^{-1}\rho_g\varphi}\in G\times\{1\}$ and $(y,2)^{\varphi^{-1}\rho_g\varphi}\in G\times\{2\}$.
Then
\begin{eqnarray*}
 \{(x,1),(y,2)\}\in E(\Gamma) &\Leftrightarrow& \exists s\in S;~y=sx\\
 &\Leftrightarrow&\{(x,1),(y,2)\}\notin E(\Sigma)\\
& \Leftrightarrow &\{(x,1)^{\varphi^{-1}\rho_g\varphi},(y,2)^{\varphi^{-1}\rho_g\varphi}\}\notin E(\Sigma)\\
 & \Leftrightarrow &\exists s\in S,\exists h\in G~;~(x,1)^{\varphi^{-1}\rho_g\varphi}=(h,1), \\
 &&(y,2)^{\varphi^{-1}\rho_g\varphi}=(sh,2)\\
 & \Leftrightarrow & \{(x,1)^{\varphi^{-1}\rho_g\varphi},(y,2)^{\varphi^{-1}\rho_g\varphi}\}\in E(\Gamma),
 \end{eqnarray*}
 which means that $\varphi^{-1}R_G\varphi\leq\Aut(\Gamma)$. Since $V(\Gamma)=V(\Sigma)$ and $\Gamma$ is
 a $\BCI$-graph, \cite[Theorem C]{arezoomand1}
 implies that $\varphi^{-1}R_G\varphi=\theta^{-1}R_G\theta$ for some $\theta\in\Aut(\Gamma)$ and $S^{-1}=gS^\alpha$ for some
 $g\in G$ and $\alpha\in\Aut(G)$. Then $
 (G\setminus S)^{-1}=G\setminus S^{-1}=G\setminus gS^\alpha=g(G\setminus S^\alpha)=g(G\setminus S)^\alpha$.

Now we claim that
$\{G\times\{1\},G\times\{2\}\}$ is $\theta$-invariant. If $(1,1)^\theta=(x,1)$ for some $x\in G$, then for all $g\in G$ we have
$(g,1)^\theta=(1,1)^{\rho_g\theta}=(x,1)^{\theta^{-1}\rho_g\theta}=(x,1)^{\varphi^{-1}\rho_h\varphi}\in G\times\{1\}$, for some
$h\in G$, which proves our claim in this case. If $(1,1)^\theta=(x,2)$ for some $x\in G$, then for all $g\in G$ we have
$(g,1)^\theta=(1,1)^{\rho_g\theta}=(x,2)^{\theta^{-1}\rho_g\theta}=(x,2)^{\varphi^{-1}\rho_h\varphi}\in G\times\{2\}$,
for some $h\in G$, which completes the proof of our claim.

Finally, $\theta\in\Aut(\Sigma)$. To see this, we have
\begin{eqnarray*}
 \{(x,1),(y,2)\}\in E(\Sigma)&\Leftrightarrow &\exists t\in G\setminus S;~ y=tx\\
 &\Leftrightarrow & \{(x,1),(y,2)\}\notin E(\Gamma)\\
 &\Leftrightarrow & \{(x,1)^\theta,(y,2)^\theta\}\notin E(\Gamma)\\
 &\Leftrightarrow & \{(x,1)^\theta,(y,2)^\theta\}\in E(\Sigma).%~\textrm{(since $\{G\times\{1\},G\times\{2\}\}$ is $\theta$-invariant)}
\end{eqnarray*}
Hence $\theta\in\Aut(\Sigma)$ and the proof is complete.
\end{proof}

\begin{lemma}\label{prime}
Let $\Gamma=\BCay(G,S)$ be connected and $|S|<p$, where $p$ is a
prime. Then $p$ does not divide the order of any stabilizer of $A$
in $V(\Gamma)$, where $A=\Aut(\Gamma)$.
\end{lemma}
\proof Suppose, towards a contradiction, that $p$ divides $|A_{(1,1)}|$.
Then there exists $x\in A_{(1,1)}$ of prime order $p$. This implies
that $\langle x\rangle$ acts on the neighbor set of $(1,1)$. Hence
$|\langle x\rangle: \langle x\rangle_{(s,2)}|\leq |S|$ for all $s\in
S$. If there exists $s\in S$ such that $\langle x\rangle_{(s,2)}=1$
then $p\leq |S|$, which is a contradiction. Hence for all $s\in S$
we have $\langle x\rangle_{(s,2)}\neq 1$.

Let $s\in S$. Then there exists $x^i\in A_{(s,2)}$, for some $1\leq
i\leq p-1$. Since $(i,p)=1$, we have $x\in A_{(s,2)}$. Again, this
implies that $\langle x\rangle$ acts on the neighbor set of $(s,2)$
and for all $t\in S$, $\langle x\rangle_{(t^{-1}s,1)}\neq 1$.
Repeating this argument, the connectivity of $\Gamma$ implies that
$x$ fixes all vertices of $\Gamma$ i.e $x=1$, a contradiction.

By a similar argument, one can see that $p$ does not divide $|A_{(1,2)}|$. If
$\Gamma$ is vertex-transitive, then all point-stabilizers of $A$ are
conjugate, which proves the result. If $\Gamma$ is not
vertex-transitive, then $A$ acts on both of sets $G\times\{1\}$ and
$G\times\{2\}$, transitively. Hence for all $g\in G$,
$|A_{(g,1)}|=|A_{(1,1)}|$ and $|A_{(g,2)}|=|A_{(1,2)}|$.

\begin{lemma}\label{2p}
Let $G=\langle g\rangle $ be a cyclic group of order $2p$ where $p$ is an odd prime
and $S$ be a subset of $G$ of size $p,$
and $X=\BCay(G,S)$. If $p$ divides the stabilizer of $(1,1)$ in $\Aut(X)$,
then $S=\langle g^2\rangle$ or $\langle g^2\rangle g$.
%is all the odd powers of $g$ or all the even powers of $g$.
In particular $X$ is a $\BCI$-graph.
\end{lemma}
\begin{proof} The assumption $p$ divides $|A_{(1,1)}|$ implies that there
is $\alpha\in A_{(1,1)}$
such that $|\alpha|= p.$ Thus $\alpha$ is the product of cycles of length $p$. Suppose toward a
contradiction that, $g^{2k}, g^{2k'+1}\in S$ and $k, k'<p$. As
$(1,1)^{\alpha}=(1,1)$,  $\alpha$ acts on the neighbors of $(1,1)$.
So we may assume that $(g^{2k}, 2)^{\alpha}=(g^{2k'+1}, 2)$. It is
easy to check that $v_1=(g^{2k}, 2)$ has a neighbor that is not a
neighbor of  $v_2=(g^{2k'+1}, 2)$. Thus $\alpha$ can not fixes all
neighbors of $v_1$. In the other hand $\alpha$ fixes $(1,1)$ and it
maps the neighbors of $v_1$ to the neighbors of $v_2$. So it has a cycle
of length less than $p$ and it is a contradiction.
%Now let $g^{2k}, g^{2k'+1}\in S$ and
%$k,k' <p, (1,1)^{\alpha}=(1,1)$ and
%the order of $\alpha$ is equal to $p$. So
%$\alpha$ acts on the neighbors of $(1,1)$ where the neighbors are
%$\{(s,2)| s \in S\}$. Hence $\alpha $ has a cycle of length $p$ say
%$((s_1,2),(s_2,2),...,(s_{p},2))$ , $s_{i}\in S$ , $ 1 \leq i \leq p$.

%Let $(g^{2k}, 2)^{\alpha}=(g^{2k'+1}, 2).$ Then
%$\alpha,$ which is a graph automorphism, maps the neighbors of $v_1=(g^{2k}, 2)$ to the neighbors of $v_2=(g^{2k'+1}, 2).$
%So there is a vertex say $v'_1$ such that $v'_1\sim v_1$ and $v'_1\nsim v_2$ and there is a vertex say $v'_2$
%such that $v'_2\nsim v_1$ and $v'_2\sim v_2.$ Two vertices $v_1$ and $v_2$ are the neighbors of $(1,1)$ and $\alpha$ fixes these
%vertices. If $\alpha$ maps the neighbors of $v_1$ to the neighbors of $v_2,$ then because of
%vertices $v'_1$, $v'_2$, $\alpha$ has cycle of length $t$ that $1<t<p$ and it is a
%contradiction by order of $\alpha$.

Let $X= \BCay(G,S)\cong \BCay(G,T).$ Then by the previous argument we
conclude that, $T$ is all the odd powers of $g$ or all the even power of $g$. Therefore $S=T$ or
$S=gT$. Thus $X$ is a $\BCI$-graph.
\end{proof}
\begin{theorem}\label{z2p}
The group $\mathbb{Z}_{2p}$, where $p$ is an odd prime, is
a $\BCI$-group.
\end{theorem}

\begin{proof} Let $G=\langle a\rangle\cong\mathbb{Z}_{2p}$ and
$\Gamma=BCay(G,S)$. By \cite[Lemma 2.8]{jin2} $\BCay(G,S)\cong
\frac{|G|}{|\langle SS^{-1}\rangle|} \BCay(\langle
SS^{-1}\rangle,S)$. If $\langle SS^{-1}\rangle =\langle a^p\rangle
\cong \mathbb{Z}_2$, then it is obvious that $\BCay(G,S)$ is a
$\BCI$-graph. Let $\langle SS^{-1}\rangle =\langle a^2\rangle
\cong\mathbb{Z}_p$. Now, we may assume that $H=\langle a^2\rangle$
which implies that $\BCay(G,S)\cong 2\BCay(H,S)$.

Let $\BCay(G,S)\cong\BCay(G,R)$, for some $R\subseteq G$. Then
$\BCay(G,R)\cong 2\BCay(K,R)$, where $K=\langle
RR^{-1}\rangle\cong\mathbb{Z}_p$, which implies that $H=K$ and
$\BCay(H,T)\cong\BCay(H,R)$. On the other hand, by \cite[Corollary
4.9]{arezoomand1}, $H$ is a $\BCI$-groups, which means that there
exists $h\in H$ and $\sigma\in\Aut(H)$ such that $R=hT^\sigma$. Now
the map
\begin{eqnarray*}
\overline{\sigma}:~~~ G&\rightarrow& G\\
a^{pi+2j}&\mapsto& a^{pi}(a^{2j})^\sigma,~~i=0,1,~j=0,1,\ldots,p-1,
\end{eqnarray*}
is an automorphism of $G$. This means that $\BCay(G,S)$ is a
$BCI$-graph.

If $\langle SS^{-1}\rangle =G$, then $\BCay(G,S)$ is a a connected
bi-Cayley graph. Suppose, towards a contradiction, that $G$ is not a
$BCI$-group. Then \cite[Example 4.5]{arezoomand1} implies that $p^2$
divides $|A|$, where $A=\Aut(\Gamma)$. Now, by \cite[Lemma
4.8]{arezoomand1}, $\Gamma$ is a Cayley graph, which implies that
$p$ divides the size of any point-stabilizer of $A$. So, by Lemma
\ref{prime}, $|S|\geq p$. Note that, by \cite[Lemma
1.1]{arezoomand1}, we may assume that $1\in S$.

Let $T=G\setminus S$ and $\Sigma=\BCay(G,T)$. Then, by Lemma
\ref{comp}, $\Sigma$ is not a $\BCI$-graph. Again, by \cite[Example
4.5 and Lemma 4.8]{arezoomand1}, $p$ divides the size of any point-stabilizer of
the automorphism group of $\Sigma$. If $|T|=p$ then by Lemma
\ref{2p}, $\Sigma=\BCay(G,T)$ is a $\BCI$-graph and it is a
contradiction. If $|T|<p$, Lemma \ref{prime} implies that $\Sigma$
is disconnected. Then $\langle TT^{-1}\rangle=\langle a^p\rangle$ or
$\langle TT^{-1}\rangle=\langle a^2\rangle$. As we discussed above
$\Sigma$ is a $\BCI$-graph, a contradiction.
\end{proof}
%**************************************************************************************************
\section{Dihedral $\BCI$-groups}
Let $D_{2n}$, $n\geq 2$, be a dihedral group of order $2n$. By \cite[Corollary 4.15]{Mil} groups of order $2p$ where $p$ is a
prime are $\mathbb{Z}_{2p}$ or $D_{2p}$. In this section we characterize
dihedral groups that they are $\BCI$-groups.
%\textcolor[rgb]{0.00,0.00,0.98}{I think we should remove the sentences :"According to Theorem
%\ref{t1}, every finite $\BCI$-group is a $\CI$-group. Thus the results
%of this section help us to classify dihedral $\CI$-groups too.", because we prove that $D_4$, $D_6$ and $D_10$ are the only $\BCI$ dihedral groups and this can not help to classifying dihedral $\CI$-groups.}
In order to achieve the goal of this section, at first we need to prove some Lemmas.

\begin{lemma}\label{4} $D_{10}$ is a $4$-$\BCI$-group.
\end{lemma}
\proof
By \cite[Lemma 2.4]{jin2}, $D_{10}$ is a 3-$\BCI$ group. Let $G=\langle a,b\mid a^5=b^2=(ab)^2=1\rangle\cong D_{10}$,
$\empty\neq S\subseteq G$, $|S|=4$ and $\Gamma=\BCay(G,S)$.
We know that $\BCay(G,S)\cong\BCay(G,gS^\alpha)$ for all $g\in G$
and $\alpha\in\Aut(G)$. Hence we may assume that $1\in S$ i.e $S=\{1,x,y,z\}$, for some $x,y,z\in G$. On the other hand, $\Aut(G)=\{\sigma_{s,l}\mid 1\leq s\leq 4, 0\leq l\leq 4\}$,
where $a^{\sigma_{s,l}}=a^s$ and $b^{\sigma_{s,l}}=a^{-l}b$. We deal with the following cases:

Case 1. $S\subseteq \langle a\rangle$. Let $S_1=\{1,a,a^2,a^3\}$.
Then $S_1^{\sigma_{2,0}}=\{1,a,a^2,a^4\}$, $S_1^{\sigma_{4,0}}=\{1,a^2,a^3,a^4\}$ and $S_1^{\sigma_{3,0}}=\{1,a,a^3,a^4\}$.
Hence, in this case, we may assume that $S=S_1=\{1,a,a^2,a^3\}$.

Case 2. $|S\cap \langle a\rangle |=3$. Then $S=\{1,a^i,a^j,a^rb\}$, where $1\leq i,j\leq 4$, $i\neq j$ and $0\leq r\leq 4$.
%In this case, we have
%$30$ possibilities.
Since there exists $\sigma\in\Aut(G)$ such that $(a^i)^\sigma=a$, we may assume that $S=\{1,a,a^m,a^nb\}$
for some $2\leq m\leq 4$ and $0\leq n\leq 4$. Let $S_2=\{1,a,a^2,b\}$. Then $S_2=\{1,a,a^3,a^nb\}^{\sigma_{2,2n}}$, where
$2n$ is computed modulo $5$. Furthermore, $S_2=a\{1,a,a^4,a^nb\}^{\sigma_{1,n+1}}$, where $n+1$ is computed modulo $5$.
Hence, in this case, we may assume that $S=S_2=\{1,a,a^2,b\}$.

Case 3. $|S\cap \langle a\rangle|=2$. Then $S=\{1,a^i,a^jb,a^rb\}$ , where $1\leq i\leq 4$, $0\leq j,r\leq 4$ and $j\neq r$.
Again, since there exists $\sigma\in\Aut(G)$ such that $(a^i)^\sigma=a$, we may assume
that $S=\{1,a,a^mb,a^nb\}$, for some $0\leq m<n\leq 4$. Let $S_3=\{1,a,b,ab\}$ and $S_4=\{1,a,b,a^2b\}$. Then
$\{1,a,a^mb,a^{m+1}b\}^{\sigma_{1,m}}=S_3$, $m=0,1,2,3$. Also $\{1,a,a^m,a^{m+2}\}^{\sigma_{1,m}}=S_4$, $m=0,1,2$.
Furthermore, $\{1,a,a^mb,a^{m+3}b\}^{\sigma_{1,m-2}}=S_4$, where $m=0,1$ and $\{1,a,b,a^4b\}^{\sigma_{1,4}}=S_3$.

Case 4. $|S\cap \langle a\rangle|=1$. Then $S=\{1,a^ib,a^jb,a^rb\}$, where $0\leq i,j,r\leq 4$ and $k\neq i\neq j\neq k$.
Since $(a^ib)^{\sigma_{1,i}}=b$, we may assume that $S=\{1,b,a^mb,a^nb\}$ for some
$1\leq m<n\leq 4$. Let $S_1=\{1,a,a^2,b\}$ as defined in Case 1. Then $b\{1,b,ab,a^2b\}^{\sigma_{4,0}}=b\{1,b,a^2b,a^3b\}^{\sigma_{2,1}}=b\{1,b,a^3b,a^4b\}=b\{1,b,ab,a^3b\}^{\sigma_{3,0}}=
b\{1,b,ab,a^4b\}^{\sigma_{1,1}}=b\{1,b,a^2b,a^4b\}^{\sigma_{3,2}}=S_1$. This shows that we may omit this case.

From the above cases, we may assume that $S$ is one of the sets $S_1$, $S_2$, $S_3$ or $S_4$. Let $\Gamma_i=\BCay(G,S_i)$.
We claim that for $i\neq j$, $\Gamma_i\ncong\Gamma_j$. We have $\Gamma_1$ is disconnected and
$\Gamma_i$, $i\neq 1$, is connected. Hence $\Gamma_1\ncong\Gamma_2,\Gamma_3,\Gamma_4$. To complete the proof
it is enough  to prove that $\Gamma_3\ncong\Gamma_2\ncong\Gamma_4\ncong\Gamma_3$.
By \cite[Theorem 2.1]{arezoomand3} or \cite[Theorem 6]{arezoomand2}, and using a simple calculation, we find that
$0$ is an eigenvalue of $\Gamma_3$ with multiplicity 10 and it is an eigenvalue of $\Gamma_4$ with multiplicity $2$,
while it is not an eigenvalue of $\Gamma_2$.
This proves that $\Gamma_3\ncong\Gamma_2\ncong\Gamma_4\ncong\Gamma_3$, which completes the proof.

\begin{lemma}\label{5} $D_{10}$ is a 5-$\BCI$ group.
\end{lemma}
\begin{proof} By Lemma \ref{4}, $D_{10}$ is a 4-$\BCI$ group.
Let $G=\langle a,b\mid a^5=b^2=(ab)^2=1\rangle\cong D_{10}$, $S\subseteq G$, $1\in S$, $|S|=5$ and
$\Gamma=\BCay(G,S)$.
%There are 126 possibilities for $S$.
We deal with the following cases:

Case 1. $S\subseteq \langle a\rangle$. Then $S=S_1=\langle a\rangle$.

Case 2. $|S\cap \langle a\rangle|=4$. Since for each $1\leq i\leq 4$ there exists $\sigma\in\Aut(G)$ such that $(a^i)^\sigma=a$, we
may assume that $a\in S$. Hence $S=\{1,a,a^i,a^j,a^kb\}$ for some $2\leq i<j\leq 4$ and $0\leq k\leq 4$. We have
$S^{\sigma_{1,k}}=\{1,a,a^i,a^j,b\}$. Furthermore, $\{1,a,a^2,a^3,b\}=\{1,a,a^2,a^4,b\}^{\sigma_{3,0}}$ and
$\{1,a,a^2,a^3,b\}=\{1,a,a^3,a^4,b\}^{\sigma_{2,0}}$.
Hence we may assume that, in this case, $S=S_2=\{1,a,a^2,a^3,b\}$.

Case 3. $|S\cap\langle a\rangle|=3$. By a similar argument to the previous case, we may assume that $a\in S$. Hence
$S=\{1,a,a^i,a^jb,a^kb\}$ for some $2\leq i\leq 4$ and $0\leq j<k\leq 4$. Since $S^{\sigma_{0,j}}=\{1,a,a^i,b,a^{k-j}b\}$, we may assume
that $S=\{1,a,a^i,b,a^kb\}$ for some
$2\leq i\leq 4$ and $1\leq k\leq 4$. Let  $S_{i,k}=\{1,a,a^i,b,a^kb\}$. Then
\begin{eqnarray*}
S_{2,1}=S_{2,4}^{\sigma_{1,4}}=S_{3,2}^{\sigma_{2,4}}=S_{3,3}^{\sigma_{2,0}}=aS_{4,1}^{\sigma_{0,1}}=aS_{4,4}^{\sigma_{0,0}}\\
S_{2,2}=S_{2,3}^{\sigma_{0,3}}=S_{3,1}^{\sigma_{2,0}}=S_{3,4}^{\sigma_{2,3}}=aS_{4,2}^{\sigma_{0,1}}=aS_{4,3}^{\sigma_{0,4}}.
\end{eqnarray*}
Hence, we may assume that $S$ is one of the sets $S_3=\{1,a,a^2,b,ab\}$ or $S_4=\{1,a,a^2,b,a^2b\}$.

Case 4. $|S\cap\langle a\rangle|=2$. By a similar argument to the previous case, we may assume that $a\in S$. Hence
$S=\{1,a,a^ib,a^jb,a^kb\}$ for some $0\leq i<j<k\leq 4$. Since $S^{\sigma_{0,i}}=\{1,a,b,a^{j-i}b,a^{k-i}b\}$, we may
assume that $S=S_{i,j}=\{1,a,b,a^ib,a^jb\}$ for some $1\leq i<j\leq 4$. On the other hand, we have
$S_{1,3}=S_{2,3}^{\sigma_{1,2}}=S_{2,4}^{\sigma_{1,4}}$,
$b\{1,a,b,ab,a^2b\}^{\sigma_{4,0}}=b\{1,a,b,ab,a^4b\}^{\sigma_{4,1}}=S_3$ and $b\{1,a,b,ab,a^3b\}^{\sigma_{3,0}}=S_4$,
where $S_3$ and $S_4$ are defined
in Case 3. Hence we may omit this case.

Case 5. $|S\cap\langle a\rangle|=1$. Then $S=\{1,a^ib,a^jb,a^mb,a^nb\}$ for some $0\leq i<j<m<n\leq 4$.
Since $S^{\sigma_{0,i}}=\{1,b,a^{j-i}b,a^{m-i}b,a^{n-i}b\}$, we may assume that $S=\{1,b,a^ib,a^jb,a^kb\}$ for
some $1\leq i<j<k\leq 4$. Furthermore, $\{1,b,ab,a^2b,a^3b\}^{\sigma_{0,3}}=\{1,b,a^2b,a^3b,a^4b\}$,
$\{1,b,ab,a^2b,a^3b\}^{\sigma_{2,0}}=\{1,b,ab,a^2b,a^4b\}$ and $\{1,b,ab,a^2b,a^3b\}^{\sigma_{3,0}}=\{1,b,ab,a^3b,a^4b\}$.
Furthermore, $\{1,b,ab,a^2b,a^3b\}=bS_2$, where $S_2$ is defined in Case 2. Hence we may omit this case.

Thus we may assume that $S$ is one of the above sets $S_1, S_2, S_3$
or $S_4$. Let $\Gamma_i=\BCay(G,S_i)$, $i=1,\ldots,4$. We shall
prove that $\Gamma_i\ncong\Gamma_j$ for all $i\neq j$. Since
$\Gamma_1$ is disconnected and $\Gamma_i$, $i\neq 1$ is connected,
we have $\Gamma_1\ncong\Gamma_2, \Gamma_3, \Gamma_4$.  By
\cite[Theorem 2.1]{arezoomand3}, $\Spec(\Gamma)=\{\pm 5,\pm 3,(\pm
2)^{[4]}, 0^{[8]}\}$, integer eigenvalues of $\Gamma_3$ are $\pm 5,
\pm 1, 0^{[8]}$ and integer eigenvalues of $\Gamma_4$ are $\pm 5,(\pm 1)^{[5]}$, which imply that
$\Gamma_3\ncong\Gamma_2\ncong\Gamma_4\ncong\Gamma_3$. This completes the proof.
\end{proof}

\begin{theorem}\label{D2p}
Let $n\geq 2$. Then $D_{2n}$ is a $\BCI$-group if and only if $n\in\{2,3,5\}$.
\end{theorem}
\begin{proof} Let $D_{2n}=\langle a, b\mid a^n=b^2=(ab)^2=1\rangle$ be a $\BCI$-group. First
let $n=4$, $S=\{1,a^2\}$ and $T=\{1,b\}$. Then $\BCay(D_8,S)\cong 4 C_4\cong\BCay(D_8,T)$. By our assumption,
there exists $g\in D_8$ and $\alpha\in\Aut(D_8)$ such that $T=gS^\alpha$, which implies that $(a^2)^\alpha=b$ a contradiction.
Hence $n\geq 4$. Suppose towards a contradiction that $n\geq 6$.
Since every $\BCI$-graph is vertex-transitive, \cite[Remark1]{EP} implies that $n\neq 6, 7$. Hence $n>7$.
On the other hand, by \cite[Proposition 11]{EP}, there exists a subset $S$ of length 7 of $D_{2n}$
such that $\Aut(\BCay(D_{2n},S))\cong D_{2n}$. Again, transitivity of $\BCay(D_{2n},S)$
implies that $4n$ divides $2n$, a contradiction. Hence it is proved that $n\in\{2,3,5\}$.

Conversely suppose that $n\in\{2,3,5\}$. We will prove that $D_4, D_6$ and $D_{10}$ are $\BCI$-groups.
$(1)$ Since $D_4$ is isomorphic to $\mathbb{Z}_2\times\mathbb{Z}_2$, it is a 3-$\BCI$ group by \cite[Lemma 2.4]{jin2}.
Clearly $\BCay(G,G)$ is a $\BCI$-graph for any group $G$. So $D_4$ is a 4-$\BCI$ group which means that it is a $\BCI$-group.

$(2)$ $D_6$ is a $\BCI$-group by \cite{jin4}.

$(3)$ Let $S\subseteq D_{10}$ and $\Gamma=\BCay(G,S)$. If $|S|\leq 5$ then $\Gamma$ is a
$\BCI$-graph, by Lemma \ref{5}. If $|S|>5$ then $|D_{10}-S|\leq 4$. Now Lemmas \ref{comp} and \ref{4}
imply that $\Gamma$ is a $\BCI$-graph.

Hence the proof is complete.
\end{proof}

%\end{document}
\bigskip
\bigskip
{\footnotesize \pn{\bf Asieh Sattari}\; \\ {Department of
Mathematical Science},\\ {Yazd University,  89195-741,} { Yazd, I. R. Iran}\\
{\tt Email: a1sattari@yahoo.com}\\

{\footnotesize \pn{\bf Majid Arezoomand}\; \\ {Department of
Engineering},\\ {University of Larestan, 74317-16137,} { Lar, I. R. Iran}\\
{\tt Email: arezoomand@lar.ac.ir}\\

{\footnotesize \pn{\bf Mohammad~A.~Iranmanesh}\; \\ {Department of
Mathematical Science},\\ {Yazd University,  89195-741,} { Yazd, I. R. Iran}\\
{\tt Email: iranmanesh@yazd.ac.ir}\\

\end{document}